\documentclass[10pt]{amsart}
\usepackage{amsmath}
\usepackage{graphics,graphicx}
\usepackage{caption} 
\usepackage{amsfonts}
\usepackage{epsfig}
\usepackage{verbatim}
\usepackage{amssymb}
\usepackage{latexsym}
\usepackage{amsmath, amsthm, amssymb}
\usepackage{color}
\newtheorem{corollary}{Corollary}[section]

\newtheorem{lemma}[corollary]{Lemma}
\newtheorem{proposition}[corollary]{Proposition}
\newtheorem{remark}[corollary]{Remark}
\newtheorem{theorem}[corollary]{Theorem}
\newcommand{\mylabel}[1]{\label{#1}
            \ifx\undefined\stillediting
            \else \fbox{$#1$}\fi }
\newcommand{\BE}{\begin{equation}}

\newcommand{\EEQ}{\end{equation}}
\newcommand{\rfb}[1]{\mbox{\rm
   (\ref{#1})}\ifx\undefined\stillediting\else:\fbox{$#1$}\fi}

\newfont{\Blackboard}{msbm10 scaled 1200}
\newcommand{\bl}[1]{\mbox{\Blackboard #1}}
\newfont{\roma}{cmr10 scaled 1200}

\def\CC{\rm \hbox{C\kern-.56em\raise.4ex
         \hbox{$\scriptscriptstyle |$}\kern+0.5 em }}

\newcommand{\rline}  {{\bl R}}

\newcommand{\mm}    {{\hbox{\hskip 0.5pt}}}

\newcommand{\bluff} {{\hbox{\raise 15pt \hbox{\mm}}}}
\newcommand{\dd}     {{\rm d}}

\makeatletter
\def\section{\@startsection {section}{1}{\z@}{-3.5ex plus -1ex minus
    -.2ex}{2.3ex plus .2ex}{\large\bf}}
\makeatother
\def\be{\begin{equation}}
\def\ee{\end{equation}}

\def\ds{\displaystyle}
\input{epsf}

\begin{document}
\thispagestyle{empty}
\title{Stability of the wave equations on a tree with local Kelvin-Voigt damping}
\author{Ka\"{\i}s Ammari}
\address{UR Analysis and Control of PDEs, UR 13ES64, Department of Mathematics, Faculty of Sciences of Monastir, University of Monastir, Tunisia}
\email{kais.ammari@fsm.rnu.tn} 
\author{Zhuangyi Liu}

\address{Department of Mathematics and Statistics, University of Minnesota, Duluth, MN 55812-3000, United States} 
\email{zliu@d.umn.edu}

\author{Farhat Shel}
\address{UR Analysis and Control of PDEs, UR 13ES64, Department of Mathematics, Faculty of Sciences of Monastir, University of Monastir, Tunisia}
\email{farhat.shel@ipeit.rnu.tn} 

\begin{abstract} 
In this paper we study the stability problem of a tree of elastic strings with local Kelvin-Voigt damping on some of the edges. Under the compatibility condition of displacement and strain and continuity condition of damping coefficients at the vertices of the tree, exponential/polynomial stability are proved. Our results generalizes the cases of single elastic string with local Kelvin-Voigt damping in \cite{lr, liu111, amregvalmer}. 
\end{abstract}


\maketitle

\tableofcontents
%
%
\section{Introduction} \label{secintro}

\setcounter{equation}{0}
In this paper, we investigate the asymptotic stability of a tree of elastic strings with local Kelvin-Voigt damping. We first introduce some notations needed to formulate the problem under consideration. Let ${\mathcal T}$ be a tree ( i.e. ${\mathcal T}$ is a planar connected graph without closed paths). 
\begin{enumerate}
	\item[] degree of a vertex -- number of incident edges at that vertex
	\item[] ${\mathcal R}$ -- root of ${\mathcal T}$, a designated vertex with degree 1
	\item[] exterior vertex -- vertex with degree 1
	\item[] interior vertex -- vertex with degree greater than 1
	\item[] $e$ -- the edge incident the root ${\mathcal R}$
	\item[] ${\mathcal O}$ -- the vertex of ${\mathcal T}$ other than ${\mathcal R}$ 
	\item[] $\bar{\alpha}$ -- multi-index of length $k$, $\bar{\alpha}=(\alpha_1, \cdots, \alpha_k)$
	\item[] ${\mathcal O}_{\bar{\alpha}}$ -- vertex of  ${\mathcal T}$ with index  ${\bar{\alpha}}$
	\item[] $e_{\bar{\alpha}}$ -- edge of ${\mathcal T}$ with index ${\bar{\alpha}}$
	\item[] ${\mathcal M}$ -- set of the interior vertices of ${\mathcal T}$
	\item[] ${\mathcal S}$ -- set of the exterior vertices of ${\mathcal T}$, excluding ${\mathcal R}$
	\item[] $I_{{\mathcal M}}$ -- index set of ${\mathcal M}$
	\item[] $I_{{\mathcal S}}$ -- index set of ${\mathcal S}$
\end{enumerate} 

We choose empty index for edge $e$ and vertex ${\mathcal O}$ . 
Assume there are $m_{\bar{\alpha}}$ edges, different from $e_{\bar{\alpha}},$
that branch out from ${\mathcal O}_{\bar{\alpha}},$ we denote these edges by $e_{\bar{\alpha} o \beta}, \, \beta = 1,...,m_{\bar{\alpha}}$ and the other vertex of the edge $e_{\bar{\alpha} o \beta}$ by ${\mathcal O}_{\bar{\alpha} o \beta}$,
i.e. the interior vertex ${\mathcal O}_{\bar{\alpha}}$, contained in the edge $e_{\bar{\alpha}},$ has multiplicity equal to $m_{\bar{\alpha}} + 1$. 

Furthermore, length of the edge $e_{\bar{\alpha}}$ is denoted by $\ell_{\bar{\alpha}}$. Then, $e_{\bar{\alpha}}$ 
may be parametrized by its arc length by means of the functions $\pi_{\bar{\alpha}}$, defined in $[0, \ell_{\bar{\alpha}}]$ 
such that $\pi_{\bar{\alpha}}(\ell_{\bar{\alpha}}) = {\mathcal O}_{\bar{\alpha}}$ and $\pi_{\bar{\alpha}}(0)$ is the other vertex of this edge. 

\begin{center}
	\includegraphics[scale=0.80]{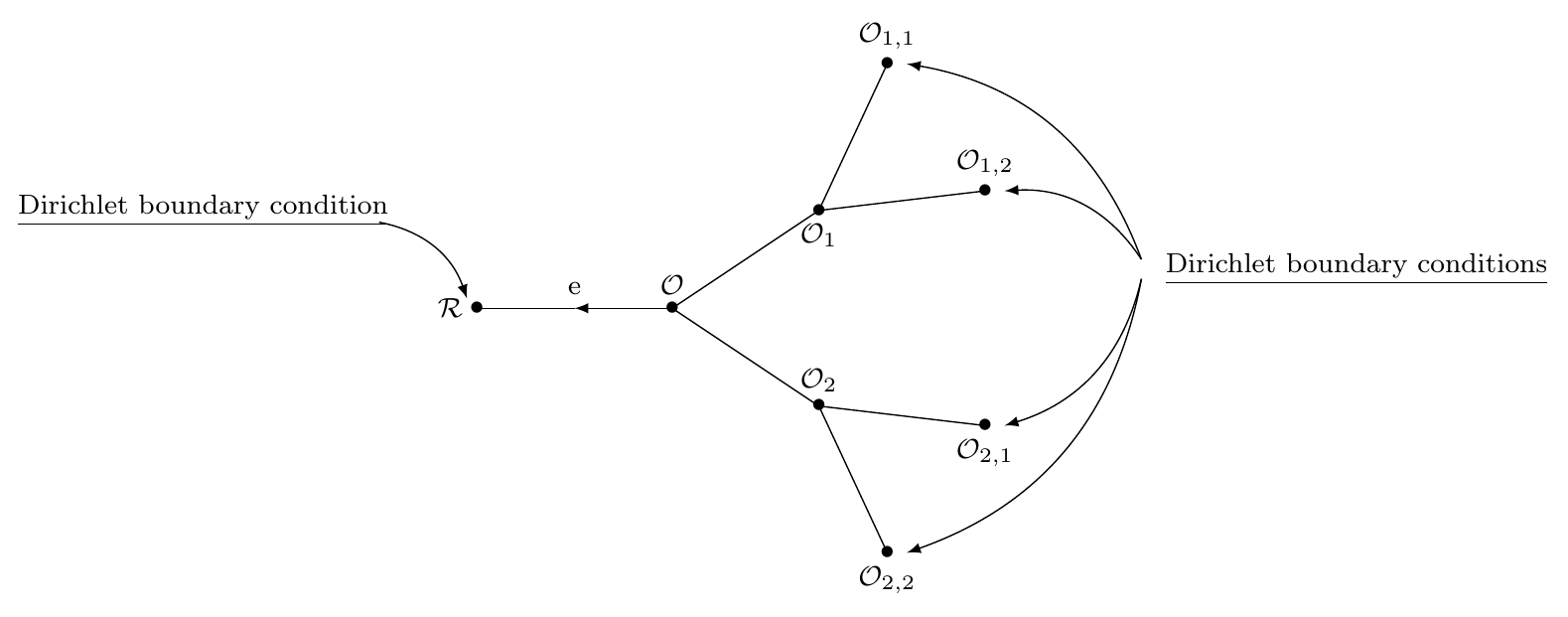} 
	\captionof{figure}{A Tree-Shaped network} 
\end{center} 

Now, we are ready to introduce a planar tree-shaped network of $N$ elastic strings, where $N \geq 2$, see \cite{lagnese,liu1,liu2, hassine, liubis} and \cite{dagerzuazua} concerning the model.  
More precisely, we consider the following initial and boundary
value problem\,: 

\be
\label{ch8eq18}
\frac{\partial^2
u_{\bar{\alpha}}}{\partial t^2}(x,t) - \, \frac{\partial }{\partial x} \left( \frac{\partial u_{\bar{\alpha}}}{\partial x} + a_{\bar{\alpha}}(x) \, 
\frac{\partial^2 u_{\bar{\alpha}}}{\partial x \partial t} \right) (x,t) = 0, \quad
0<x<\ell_{\bar{\alpha}},\ t>0, \, \bar{\alpha} \in I := I_{\mathcal M}\cup I_{\mathcal S}, \ee \be
\label{ch8eq28} u(0,t)= 0, \, u_{\bar{\alpha}}(\ell_{\bar{\alpha}},t)=0,
\, \bar{\alpha} \in I_{{\mathcal S}}, \, t >0, \ee 
\be 
\label{ch8inh}
u_{\bar{\alpha} \circ \beta}(0,t)=
u_{\bar{\alpha}}(\ell_{\bar{\alpha}},t), \quad t>0, \, \beta =
1,2,..., m_{\bar{\alpha}}, \, \bar{\alpha} \in I_{{\mathcal M}}, 
\ee 
\be
\label{ch8cont} \ds \sum_{\beta = 1}^{m_{\bar{\alpha}}}
\frac{\partial u_{\bar{\alpha} \circ \beta}}{\partial x}(0,t)+a_{\bar{\alpha} \circ \beta}(0) \frac{\partial^2 u_{\bar{\alpha} \circ \beta}}{\partial x \partial t}(0,t) =
\frac{\partial u_{\bar{\alpha}}}{\partial x}(\ell_{\bar{\alpha}},t)+a_{\bar{\alpha}}(\ell_{\bar{\alpha}}) \frac{\partial^2 u_{\bar{\alpha}}}{\partial x \partial t }(\ell_{\bar{\alpha}},t), 
\quad t>0, \, \bar{\alpha} \in I_{{\mathcal M}}, \ee 
\be 
\label{ch8eq38}
u_{\bar{\alpha}}(x, 0) = u^0_{\bar{\alpha}}(x), \, \frac{\partial u_{\bar{\alpha}}}{\partial t} (x, 0) = u^1_{\bar{\alpha}}(x), \quad
0<x<\ell_{\bar{\alpha}}, \, {\bar{\alpha}} \in I, 
\ee 
where
$u_{\bar{\alpha}} : [0,\ell_{\bar{\alpha}}] \times (0,+ \infty)
\rightarrow \rline, \, \bar{\alpha} \in I,$ be the transverse displacement with index $\bar{\alpha}$, $a_{\bar{\alpha }}\in
L^{\infty }(0,\ell _{\bar{\alpha }})$ and, either $a_{\bar{\alpha }}$ is zero, 
that is, $e_{\bar{\alpha }}$ is a purely elastic edge, or $a_{\bar{\alpha }}(x) > 0$,  on $(0,\ell _{\bar{\alpha }})$. Such edge will be called a K-V edge. This setting also includes the case that $a_{\bar{\alpha }}(x) > 0$ only on subintervals of $(0,\ell _{\bar{\alpha }})$, since we then can consider this edge as the union of pure elastic edges and K-V edges.

We assume that $\mathcal{T}$ contains at least one K-V edge. Furthermore, we
suppose that every maximal subgraph of purely elastic edges is a tree, whose
leaves are attached to K-V edges.

Models of the transient behavior of some or all of the state variables describing the motion of flexible structures have been of great interest
in recent years, for details about physical motivation for the models, see \cite{dagerzuazua}, \cite{lagnese} and the references therein. Mathematical
analysis of transmission partial differential equations is detailed in \cite{lagnese}. For the feedback stabilization problem the wave or Schr\"odinger equations in networks, we refer
the readers to references \cite{ammari1}-\cite{amjellk}, \cite{lagnese}.  

Our aim is to prove, under some assumptions on damping coefficients $a_{\bar{\alpha}}$, $\bar{\alpha} \in I$, exponential and polynomial stability results for the system \rfb{ch8eq18}-\rfb{ch8eq38}.

We define the natural energy $E(t)$ of a solution $\underline{u} = (u_{\bar{\alpha}})_{\bar{\alpha} \in I}$ of \rfb{ch8eq18}-\rfb{ch8eq38} by
\be 
\label{energy1}
E(t)=\frac{1}{2} \ds \sum_{\bar{\alpha} \in I} \int_0^{\ell_{\bar{\alpha}}} 
\left( \left|\frac{\partial u_{\bar{\alpha}}}{\partial t}(x,t)\right|^2 + \left|\frac{\partial u_{\bar{\alpha}}}{\partial x}(x,t)\right|^2 \right) \, dx.
\ee 

It is straightforward to check that every sufficiently smooth solution of \rfb{ch8eq18}-\rfb{ch8eq38} satisfies the following dissipation law 
\begin{equation}\label{dissipae1}
\frac{d}{dt} E(t) = - \ds \sum_{\bar{\alpha} \in I} \int_0^{\ell_{\bar{\alpha}}} a_{\bar{\alpha}}(x) \, \left|\frac{\partial^2 u_{\bar{\alpha}}}{\partial x \partial t}(x,t)\right|^2 \, dx \leq 0, 
\end{equation}
and therefore, the energy is a nonincreasing function of the time variable $t$.

\medskip
 
The main results of this paper then concern the precise asymptotic behavior of the solutions of \rfb{ch8eq18}-\rfb{ch8eq38}. 
Our technique is a special frequency domain analysis of the corresponding operator.

\medskip

This paper is organized as follows:
In Section \ref{wellposed}, we give the proper functional setting for system \rfb{ch8eq18}-\rfb{ch8eq38} and prove that the system is well-posed.
In Section \ref{specanal}, we analyze the resolvent of the wave operator associated to the dissipative system \rfb{ch8eq18}-\rfb{ch8eq38} and prove the asymptotic behavior of the corresponding semigroup. In the last section we give some comments on the cases of more general graph for the network. 

\section{Well-posedness of the system} \label{wellposed}
In order to study system \rfb{ch8eq18}-\rfb{ch8eq38} we need a proper functional setting. We define the following space 
$$
{\mathcal H} = V \times H
$$
where
$H = \ds \prod_{\bar{\alpha} \in I} L^2 (0,\ell_{\bar{\alpha}})$ and 
$V = \left\{\underline{u} \in \ds \prod_{\bar{\alpha} \in I} H^1(0,\ell_{\bar{\alpha}} ) \, :   u(0) = 0, \, u_{\bar{\alpha}}(\ell_{\bar{\alpha}}) = 0, \, \bar{\alpha} \in I_{\mathcal S}, \; \hbox{satisfies} \; \rfb{e5} \ \right\}$ 
\be
\label{e5}
u_{\bar{\alpha} \circ \beta}(0)=
u_{\bar{\alpha}}(\ell_{\bar{\alpha}}), \, \beta =
1,..., m_{\bar{\alpha}}, \, \bar{\alpha} \in I_{{\mathcal M}}, 
\ee
and equipped with the inner products
\be
\label{ipVb}
<\underline{u},\underline{\tilde{u}}>_{{\mathcal H}}= \ds \sum_{\bar{\alpha} \in I} \int_0^{\ell_{\bar{\alpha}}}   \left( u_{\bar{\alpha}}(x) \, \bar{\tilde{u}}_{\bar{\alpha}}(x) + u^\prime_{\bar{\alpha}}(x) \,  \bar{ \tilde{u}}^\prime{\bar{\alpha}}(x)\right)\,  dx. 
\ee
System \rfb{ch8eq18}-\rfb{ch8eq38} can be rewritten as the first order evolution equation
\begin{equation} \left\{
\begin{array}{l}
\ds \frac{\partial }{\partial t} \left( \begin{array}{c} \underline{u} \cr \frac{\partial \underline{u}}{\partial t}
 \end{array} \right) =\mathcal{A}_d 

\left( \begin{array}{c} \underline{u} \cr \frac{\partial \underline{u}}{\partial t}
 \end{array} \right),\\
\underline{u}(0)= \underline{u}^{0}, \, \ds \frac{\partial \underline{u}}{\partial t} = \underline{u}^{1}\end{array}\right.\label{pbfirstorder}
\end{equation}
where the operator $\mathcal{A}_d : {\mathcal D}({\mathcal A}_d) \subset {\mathcal H} \rightarrow {\mathcal H}$ is defined by 
$$
\mathcal{A}_d \left( \begin{array}{c} \underline{u} \\ \underline{v} \end{array} \right):= \left( \begin{array}{c}  \underline{v}  \\ ( \underline{u^{\prime}} + \underline{a} * \underline{v^{\prime}} )^\prime  \end{array} \right),
$$
with
$$
\underline{a} := (a_{\bar{\alpha}})_{\bar{\alpha} \in I} \; \;\text{and}\; \;
\underline{a} *  \underline{v^{\prime}} := (a_{\bar{\alpha}}  
v^{\prime}_{\bar{\alpha}})_{\bar{\alpha} \in I},
$$
and
$$
{\mathcal D}({\mathcal A}_d):= \left\{(\underline{u}, \underline{v}) \in \mathcal{H}, \, \underline{v} \in  V, ( \underline{u^{\prime}} + \underline{a} * \underline{v^{\prime}} ) \in \ds \prod_{\bar{\alpha} \in I} H^1(0,\ell_{\bar{\alpha}})\,: (\underline{u}, \underline{v}) \, 
\mbox {\textrm{satisfies}} \, (\ref {e6})
\right\},
$$

\begin{equation}\label{e6}
\ds \sum_{\beta = 1}^{m_{\bar{\alpha}}}
 u_{\bar{\alpha} \circ \beta}(0)+ a_{\bar{\alpha} \circ \beta}(0) v^{\prime}_{\bar{\alpha} \circ \beta}(0)  =
 u^{\prime}_{\bar{\alpha}}(\ell_{\bar{\alpha}})+ a_{\bar{\alpha}}(\ell _{\bar{\alpha}} ) v^{\prime}_{\bar{\alpha}}(\ell _{\bar{\alpha}}),
\quad \, \bar{\alpha} \in I_{{\mathcal M}}.
\end{equation}

\begin{lemma} \label{injlem}
The operator $\mathcal{A}_{d}$ is dissipative, $\left\{0,1 \right\} \subset \rho (\mathcal{A}_{d}):$
the resolvent set of $\mathcal{A}_d.$ 
\end{lemma}
\begin{proof}
For $(\underline{u},\underline{v})\in \mathcal{D}(\mathcal{A}_{d}),$ we have
\begin{equation*}
Re(\left\langle \mathcal{A}_{d}(\underline{u},\underline{v}),(\underline{u},%
\underline{v})\right\rangle _{\mathcal{H}})=Re\sum_{\bar{\alpha }\in
I}\left( \int_{0}^{\ell _{\bar{\alpha }}}v^{\prime}_{\bar{%
\alpha }}\overline{u^{\prime}_{\bar{\alpha }}}dx+\int_{0}^{\ell
_{\bar{\alpha }}}(u^{\prime}_{\bar{\alpha }}+a_{%
\bar{\alpha }}v^{\prime}_{\bar{\alpha }})^{\prime}\overline{v_{%
\bar{\alpha }}}dx\right). 
\end{equation*}

Performing integration by parts and using transmission and boundary
conditions, a straightforward calculations leads to 
\begin{equation*}
Re(\left\langle \mathcal{A}_{d}(\underline{u},\underline{v}),(\underline{u},%
\underline{v})\right\rangle _{\mathcal{H}})=-\sum_{\bar{\alpha }\in
I}\int_{0}^{\ell _{\bar{\alpha }}}a_{\bar{\alpha }}(x)\left|
v^{\prime}_{\bar{\alpha }}(x)\right| ^{2}dx\leq 0
\end{equation*}
which proves the dissipativeness of the operator $\mathcal{A}_{d}$ in $%
\mathcal{H}.$

Next, using Lax-Milgram's lemma, we prove that $1 \in \rho (\mathcal{A}_{d}).$ For
this, let $(f,g)\in \mathcal{H}$ and we look for $(\underline{u},\underline{v})\in
D(A_{d})$ such that
\begin{equation*}
(\mathcal{I}-\mathcal{A}_{d})(\underline{u},\underline{v})=(\underline{f},%
\underline{g})
\end{equation*}
which can be written as
\begin{eqnarray}
u_{\bar{\alpha }}-v_{\bar{\alpha }} &=&f_{\bar{\alpha }},\;%
\bar{\alpha }\in I,  \label{Lax1} \\
v_{\bar{\alpha }}-(u^{\prime}_{\bar{\alpha }}+a_{%
\bar{\alpha }}v^{\prime}_{\bar{\alpha }})^{\prime} &=&g_{\bar{\alpha }},\;\bar{\alpha }\in I.  \label{Lax2}
\end{eqnarray}
Let $\underline{w} \in V$; multiplying (\ref{Lax2}) by $w_{\bar{\alpha }}$, then summing over $%
\bar{\alpha }\in I$, we obtain
$$
\sum_{\bar{\alpha }\in I}\int_{0}^{\ell _{\bar{\alpha }}}v_{%
\bar{\alpha }}\overline{w_{\bar{\alpha }}}dx-\sum_{\bar{%
\alpha }\in I}\left[ \left( u^{\prime}_{\bar{\alpha }}+a_{\bar{%
\alpha }}v^{\prime}_{\bar{\alpha }}\right) (x)\overline{w_{%
\bar{\alpha }}}(x)\right] _{0}^{\ell _{\bar{\alpha }}}+ 
$$
\begin{equation}
\sum_{%
\bar{\alpha }\in I}\int_{0}^{\ell _{\bar{\alpha }}}\left( u^{\prime}_{\bar{\alpha }}+a_{\bar{\alpha }}v^{\prime}_{\bar{%
\alpha }}\right) \overline{w^{\prime}_{\bar{\alpha }}}dx   
=\sum_{\bar{\alpha }\in I}\int_{0}^{\ell _{\bar{\alpha }}}g_{\bar{\alpha }}\overline{w_{\bar{\alpha }}}
dx. 
\label{Lax3}
\end{equation}
Replacing $v_{\bar{\alpha }}$ in the last equality by (\ref{Lax1}), we get
\begin{equation}
\varphi (\underline{u},\underline{w})=\psi (\underline{w})  \label{Lax4}
\end{equation}
where
\begin{equation*}
\varphi (\underline{u},\underline{w})=\sum_{\bar{\alpha }\in I}\left(
\int_{0}^{\ell _{\bar{\alpha }}}u_{\bar{\alpha }}\overline{w_{%
\bar{\alpha }}}dx+\int_{0}^{\ell _{\bar{\alpha }}}(1+a_{\bar{%
\alpha }})u^{\prime}_{\bar{\alpha }}\overline{w^{\prime}_{%
\bar{\alpha }}}\right) 
\end{equation*}
and
\begin{equation*}
\psi (\underline{w})=\sum_{\bar{\alpha }\in I}\left( \int_{0}^{\ell _{%
\bar{\alpha }}} \left( f_{\bar{\alpha }} + g_{\bar{\alpha }} \right) \, \overline{%
w_{\bar{\alpha }}}dx+\int_{0}^{\ell _{\bar{\alpha }}}a_{\bar{%
\alpha }}f^{\prime}_{\bar{\alpha }}\overline{w^{\prime}_{%
\bar{\alpha }}}dx\right).
\end{equation*}
The fuction $\varphi $ is a continuous sesquilinear form on $V\times V$ and $%
\psi $ is a continuous anti-linear form on $V;$ here $V$ is equipped with
the inner product
\begin{equation*}
\left\langle \underline{f},\underline{g}\right\rangle =\sum_{\bar{%
\alpha }\in I}\left( \int_{0}^{\ell _{\bar{\alpha }}}u_{\bar{%
\alpha }}\overline{w_{\bar{\alpha }}}dx+\int_{0}^{\ell _{\bar{%
\alpha }}}u^{\prime}_{\bar{\alpha }}\overline{w^{\prime}_{%
\bar{\alpha }}}\right) .
\end{equation*}

Since $\varphi$ is coercive on $V,$ the conclusion is deduced by the
Lax-Milgram lemma$.$

\medskip
By the same why we prove that $0 \in \rho(\mathcal{A}_d)$.

\end{proof}

By the Lumer-Phillip's theorem (see \cite{Pazy, tucsnakbook}), we have the following proposition.

\begin{proposition}\label{3exist1} 
The operator $\mathcal{A}_{d}$ generates a $\mathcal{C} _{0}$-semigroup of contraction $(S_{d}(t))_{t \geq0}$ on the Hilbert space $\mathcal{H}$.

Hence, for an initial datum $(\underline{u}^{0},\underline{u}^{1}) \in {\mathcal H}$, there exists a unique solution $\left(\underline{u}, 
\frac{\partial \underline{u}}{\partial t} \right)\in C([0,\,+\infty),\, {\mathcal H})$
to  problem (\ref{pbfirstorder}). Moreover, if $(\underline{u}^{0},\underline{u}^{1}) \in \mathcal{D}(\mathcal{A}_d)$, then
$$\left(\underline{u},\frac{\partial \underline{u}}{\partial t} \right) \in C([0,\,+\infty),\, \mathcal{D}(\mathcal{A}_d)).$$ 
\end{proposition}

Furthermore, the solution $(\underline{u},\frac{\partial \underline{u}}{\partial t})$ of \rfb{ch8eq18}-\rfb{ch8eq38} with initial datum in $\mathcal{D}(\mathcal{A}_d)$ satisfies \rfb{dissipae1}.
Therefore the energy is decreasing.

\section{Asymptotic behaviour} \label{specanal} 

In order to analyze the asymptotic behavior of system \rfb{ch8eq18}-\rfb{ch8eq38}, we shall use the
following characterizations for exponential and polynomial stability of a $%
\mathcal{C}_{0}$-semigroup of contractions:

\begin{lemma} \cite{huang, pruss}
\label{lem31} A $\mathcal{C}_{0}$-semigroup of contractions $(e^{t\mathcal{A}})_{t \geq 0}$ defined
on the Hilbert space $\mathcal{H}$ is exponentially stable if and only if 
\begin{equation}
\mathbf{i}\mathbb{R}\subset \rho (\mathcal{A})  \label{st1}
\end{equation}
and 
\begin{equation}
\limsup_{\left| \beta \right| \rightarrow +\infty} \left\| (\mathbf{i}%
\beta \mathcal{I}-\mathcal{A})^{-1}\right\| _{\mathcal{L}(\mathcal{H}%
)}<\infty  \label{st2}
\end{equation}
\end{lemma}

\begin{lemma} \cite{BT}
\label{lem32} A $\mathcal{C}_{0}$-semigroup of contractions $(e^{t\mathcal{A}})_{t \geq 0}$ on the
Hilbert space $\mathcal{H}$ satisfies 
\begin{equation*}
\left\| e^{t\mathcal{A}}y_{0}\right\| \leq \frac{C}{t^{\frac{1}{\alpha }}}%
\left\| y_{0}\right\| _{\mathcal{D}(\mathcal{A})}
\end{equation*}
for some constant $C>0$ and for $\alpha >0$ if and only if (\ref{st1})
holds and 
\begin{equation}
\limsup_{\left| \beta \right| \rightarrow +\infty} \frac{1}{%
\beta ^{\alpha }}\left\| (\mathbf{i}\beta \mathcal{I}-\mathcal{A}%
)^{-1}\right\| _{\mathcal{L}(\mathcal{H})}<\infty .  \label{st3}
\end{equation}
\end{lemma}

\begin{lemma} [Asymptotic stability]
\label{lem33} The operator $\mathcal{A}_{d}$ verifies (\ref{st1}) and then the associated
semigroup $(S_d(t))_{t \geq 0}$ is asymptotically stable on $\mathcal{H}$.
\end{lemma}

\begin{proof}
Since $0\in\rho(\mathcal{A}_d)$ we only need here to prove that $(i\beta \mathcal{I}-\mathcal{A}_d)$ is a one-to-one correspondence in the energy space $\mathcal{H}$ for all $\beta\in\mathbb{R}^{*}$. The proof will be done in two steps: in the first step we will prove the injective property of $(i\beta \mathcal{I}-\mathcal{A}_d)$ and in the second step we will prove the surjective property of the same operator.

\begin{itemize}
\item
Suppose that there exists $\beta \in \mathbb{R}^*$ such that $Ker (\mathbf{i} \beta \mathcal{I} - \mathcal{A}_d)  \neq \left\{0 \right\}$. So
$\lambda =\mathbf{i}\beta$ is an eigenvalue of $\mathcal{A}_{d},$ then let $(%
\underline{u},\underline{v})$ an eigenvector of $\mathcal{D}(\mathcal{A}_{d})$
associated to $\lambda .$ For every $\bar{\alpha }$ in $I$ we have
\begin{eqnarray}
v_{\bar{\alpha }} &=&\mathbf{i}\beta u_{\bar{\alpha }},
\label{st4} \\
(u^{\prime}_{\bar{\alpha }}+a_{\bar{\alpha }%
}v^{\prime}_{\bar{\alpha }})^{\prime} &=&\mathbf{i}\beta v_{\bar{\alpha 
}}.  \label{st5}
\end{eqnarray}
We have

\begin{equation*}
\left\langle \mathcal{A}_{d}(\underline{u},\underline{v}),(\underline{u},%
\underline{v})\right\rangle _{\mathcal{H}}=\sum_{\bar{\alpha }\in
I}\int_{0}^{\ell _{\bar{\alpha }}}a_{\bar{\alpha }}\left| v^{\prime}_{\bar{\alpha }}\right| ^{2}dx=0.
\end{equation*}
Then $a_{\bar{\alpha }}v^{\prime}_{\bar{\alpha }}=0$ a.e. on $%
(0,\ell _{\bar{\alpha }})$.

\medskip

Let $e_{\bar{\alpha }}$ a K-V edge.
According to (\ref{st4}) and the fact that $a_{\bar{\alpha}}v^{\prime}_{\bar{\alpha}}=0$ a.e. on $(0,\ell _{\bar{\alpha}}),$ we
have $u^{\prime}_{\bar{\alpha}}=0$ a.e. on $\omega _{\bar{%
\alpha }}.$ Using (\ref{st5}), we deduce that $v_{\bar{\alpha}}=0$ on $%
\omega _{\bar{\alpha}}.$ Return back to (\ref{st4}), we conclude that $%
u_{\bar{\alpha}}=0$ on $\omega _{\bar{\alpha}}.$

Putting $y=u^{\prime}_{\bar{\alpha}}+a_{\bar{\alpha}%
}v^{\prime}_{\bar{\alpha}}=(1+\mathbf{i}\beta a_{\bar{\alpha}%
})u^{\prime}_{\bar{\alpha}},$ we have $y\in H^{2}(0,\ell _{%
\bar{\alpha}})$ and $y^{\prime}=-\beta ^{2}u_{\bar{\alpha}}.$
Hence $y$ satisfies the Cauchy problem
\begin{equation*}
y^{\prime \prime}+\frac{\beta ^{2}}{1+\mathbf{i}\beta a_{\bar{\alpha}}}y=0,\;\;y(z_{0})=0,%
\;\;y^{\prime}(z_{0})=0
\end{equation*}
for some $z_{0}$ in $\omega _{\bar{\alpha}}.$ Then $y$ is zero on $%
(0,\ell _{\bar{\alpha}})$ and hence $u^{\prime}_{\bar{\alpha}%
}$ and $u_{\bar{\alpha}}$ are zero on $(0,\ell _{\bar{\alpha}})$. Moreover $u_{\bar{\alpha}}$ and $u^{\prime}_{\bar{%
\alpha }}+a_{\bar{\alpha}}v^{\prime}_{\bar{\alpha}}$ vanish
at $0$ and at $\ell _{\bar{\alpha}}$.

If $e_{\bar{\alpha }}$ is a purely elastic edge attached to a K-V edge
at one of its ends, denoted by $x_{\bar{\alpha }},$ then $u_{%
\bar{\alpha }}(x_{\bar{\alpha }})=0,\;u^{\prime}_{\bar{%
\alpha }}(x_{\bar{\alpha }})=0.$ Again, by the same way we can deduce
that $ u^{\prime}_{\bar{\alpha }}$ and $u_{\bar{\alpha }}$ are zero in $%
L^{2}(0,\ell _{\bar{\alpha }})$ and at both ends of $e_{\bar{%
\alpha }}$. We iterate such procedure on every maximal subgraph of purely
elastic edges of $\mathcal{T}$ (from leaves to the root), to obtain finally that $(\underline{u},%
\underline{v})=0$ in $\mathcal{D}(\mathcal{A}_{d}),$ which is in
contradiction with the choice of $(\underline{u},\underline{v}).$

\item

Now given $(\underline{f},\underline{g})\in\mathcal{H}$, we solve the equation 
$$
(\mathbf{i} \beta \mathcal{I}-\mathcal{A}_d)(\underline{u},\underline{v})=(\underline{f},\underline{g})
$$
or equivalently,
\begin{equation}\label{Swave}
\left\{\begin{array}{l}
\underline{v} = \mathbf{i}\beta \underline{u} - \underline{f}
\\
\beta^{2} \underline{u} + \underline{u}^{\prime \prime} + \mathbf{i}\beta\,(\underline{a} \ast \underline{u}^\prime)^\prime =(\underline{a} \ast \underline{f}^\prime)^\prime - \mathbf{i} \beta \underline{f}- \underline{g}.
\end{array}\right.
\end{equation}
Let's define the operator
$$
A \underline{u} = - \underline{u}^{\prime \prime} - \mathbf{i} \beta \,(\underline{a} \ast \underline{u}^\prime)^\prime,\quad \forall\, \underline{u}\in V.
$$
It is easy to show that $A$ is an isomorphism from $V$ onto $V^\prime$ (where $V^\prime$ is the dual space of $V$ obtained by means of the inner product in $H$). Then the second line of \eqref{Swave} can be written as follow
\begin{equation}\label{Eqwave}
\underline{u}-\beta^{2}A^{-1} \underline{u}=A^{-1}\left(\underline{g}+ \mathbf{i} \beta \underline{f}-(\underline{a} \ast \underline{f}^\prime)^\prime\right).
\end{equation}
If $\underline{u}\in\mathrm{Ker}(\mathcal{I}-\beta^{2}A^{-1})$, then $\beta^{2}\underline{u}-A\underline{u}=0$. It follows that
\begin{equation}\label{Awave}
\beta^{2}\underline{u}+\underline{u}^{\prime \prime}+\mathbf{i}\beta (\underline{a}\ast \underline{u}^\prime)^\prime=0.
\end{equation}
Multiplying \eqref{Awave} by $\overline{\underline{u}}$ and integrating over $\mathcal{T}$, then by Green's formula we obtain 
$$
\beta^{2} \sum_{{\bar{\alpha }} \in I} \int_0^{\ell_{\bar{\alpha }}}|u_{\bar{\alpha }}(x)|^{2}\,\dd x- \sum_{{\bar{\alpha }} \in I} \int_0^{\ell_{\bar{\alpha }}}| u^\prime_{\bar{\alpha }}(x)|^{2}\,\dd x-\mathbf{i}\beta \sum_{{\bar{\alpha }} \in I} \int_0^{\ell_{\bar{\alpha }}} a_{\bar{\alpha }}(x) \, |u^\prime_{\bar{\alpha }}(x)|^{2}\,\dd x=0.
$$
This shows that 
$$
\sum_{{\bar{\alpha }} \in I} \int_0^{\ell_{\bar{\alpha }}} a_{\bar{\alpha }} (x) \, |u_{\bar{\alpha }}^\prime(x)|^{2}\,\dd x=0,
$$
which imply that $\underline{a} \ast \underline{u}^\prime =0$ in $\mathcal{T}$.
\\
Inserting this last equation into~\eqref{Awave} we get
$$
\beta^{2}\underline{u} +\underline{u}^{\prime \prime} =0,\qquad \text{in }\mathcal{T}.
$$
According to the first step, we have that $\mathrm{Ker}(\mathcal{I}-\beta^{2}A^{-1})=\{0\}$. On the other hand thanks to the compact embeddings $V\hookrightarrow H$ and $H\hookrightarrow V^\prime$ we see that $A^{-1}$ is a compact operator in $V$. Now thanks to Fredholm's alternative, the operator $(\mathcal{I}-\beta^{2}A^{-1})$ is bijective in $V$, hence the equation \eqref{Eqwave} have a unique solution in $V$, which yields that the operator $(\mathbf{i}\beta \mathcal{I}-\mathcal{A}_d)$ is surjective in the energy space $\mathcal{H}$. The proof is thus complete.

\end{itemize}

\end{proof}

Before stating the main results, we define a property (P) on $\underline{a}$ as follows
\begin{equation*}
(P)\;\;\;\forall \bar{\alpha}\in I,\;\;a_{\bar{\alpha}%
}^{\prime },a_{\bar{\alpha}}^{\prime \prime }\in L^{\infty }(0,\ell _{%
\bar{\alpha}})\;\; \text{and}\;\; \;\;\forall \bar{\alpha}\in I_{\mathcal{M}%
},\;\;a_{\bar{\alpha}
}^{\prime }(\ell _{\bar{\alpha}})- \ds \sum_{\beta =1}^{m_{\bar{\alpha}}}a_{\bar{%
\alpha }\circ \beta }^{\prime }(0)\leq 0. 
\end{equation*}

\begin{theorem} \label{th} Suppose that the function $\underline{a}$ satisfies property $(P)$, then
\begin{itemize}

\item[(i)] If $\underline{a}$ is continuous at every inner node of $\mathcal{T}$ then $(S_{d}(t))_{t \geq 0}$ is exponentially stable on $\mathcal{H}$.

\item[(ii)]  If $\underline{a}$ is not continuous at least at an inner node of $\mathcal{T}$ then $(S_{d}(t))_{t \geq 0}$ is polynomially stable on $\mathcal{H}$, in particular there exists $C>0$
such that for all $t >0$ we have
\begin{equation*}
\left\| e^{\mathcal{A}t}(\underline{u}^{0},\underline{u}^{1})\right\| _{%
\mathcal{H}}\leq \frac{C}{t^{2}}\left\| (\underline{u}^{0},\underline{u}%
^{1})\right\| _{\mathcal{D}(\mathcal{A})}, \, \forall \, (\underline{u}^{0},\underline{u}^{1})\in \mathcal{D}(\mathcal{A}).
\end{equation*}
\end{itemize}
\end{theorem}

\begin{proof}
According to Lemma \ref{lem31}, Lemma \ref{lem32}, and Lemma \ref{lem33},   it suffices to prove that for $\gamma =0$, when $\underline{a}$ is continous at every inner node, or $\gamma = 1/2$, when $\underline{a}$ is not continuous at an inner node, there exists $r>0$ such that  
\begin{equation}
\inf_{\left\| (\underline{u},\underline{v})\right\| _{\mathcal{H}}, \beta \in \mathbb{R}}\beta ^{\gamma}\left\| (\mathbf{i}\beta \mathcal{I}-\mathcal{A}_{d})(\underline{u},%
\underline{v})\right\| _{\mathcal{H}} \geq r.  \label{3.22}
\end{equation}
Suppose that (\ref{3.22}) fails. Then there exists a sequence of real numbers 
$\beta _{n}$, with $\beta _{n}\rightarrow \infty $ (without loss of
generality, we suppose that $\beta _{n}>0$ ), and a sequence of vectors $(%
\underline{u}_{n},\underline{v}_{n})$ in $\mathcal{D}(\mathcal{A}_d)$ with $%
\left\| (\underline{u}_{n},\underline{v}_{n})\right\| _{\mathcal{H}}=1$ such
that 
\begin{equation}
\beta _{n}^{\gamma}\left\| (\mathbf{i}\beta _{n}\mathcal{I}-\mathcal{A}_{d})(\underline{u}_{n},%
\underline{v}_{n})\right\| _{\mathcal{H}}\rightarrow 0.  \label{ex0}
\end{equation}

We shall prove that $\left\| (\underline{u}_{n},\underline{v}_{n})\right\| _{%
\mathcal{H}}=o(1),$ which contradict the hypotheses on $(\underline{u}_{n},%
\underline{v}_{n}).$

Writing (\ref{ex0}) in terms of its components, we get for every $\bar{%
\alpha }\in I,$%
\begin{eqnarray}
\beta _{n}^{\gamma}(\mathbf{i}\beta _{n}u_{\bar{\alpha},n}-v_{\bar{\alpha},n}) &=:&f_{%
\bar{\alpha},n}=o(1)\;\;\;\text{in }H^{1}(0,\ell _{\bar{\alpha}%
}),  \label{ex1} \\
\beta _{n}^{\gamma}(\mathbf{i}\beta _{n}v_{\bar{\alpha},n}-(u^{\prime}_{%
\bar{\alpha},n}+a_{\bar{\alpha}}v^{\prime}_{\bar{\alpha 
},n})^{\prime}) &=:&g_{\bar{\alpha},n}=o(1)\;\;\;\text{in }L^{2}(0,\ell _{%
\bar{\alpha}}).  \label{ex2}
\end{eqnarray}

Note that 
\begin{equation*}
\beta _{n}^{\gamma}\sum_{\bar{\alpha}\in I}\int_{0}^{\ell _{\bar{\alpha}}}a_{%
\bar{\alpha}}(x)\left| v^{\prime}_{\bar{\alpha}}(x)\right|
^{2}dx=Re\left( \left\langle 
\beta _{n}^{\gamma} (\mathbf{i}\beta _{n}\mathcal{I}-\mathcal{A}%
_{d})(\underline{u}_{n},\underline{v}_{n}),(\underline{u}_{n},\underline{v}%
_{n})\right\rangle _{\mathcal{H}}\right) =o(1).
\end{equation*}
Hence, for every $\bar{\alpha} \in I$ 
\begin{equation}
\beta _{n}^{\frac{\gamma}{2}}\left\| a_{\bar{\alpha}}^{\frac{1}{2}}v^{\prime}_{\bar{\alpha 
},n}\right\| _{L^{2}(0,\ell _{\bar{\alpha}})}=o(1).  \label{ex3}
\end{equation}
Then from (\ref{ex1}), we get that 
\begin{equation}
\beta _{n}^{\frac{\gamma}{2}} \left\| a_{\bar{\alpha}}^{\frac{1}{2}}\beta _{n}u^{\prime}_{%
\bar{\alpha},n}\right\| _{L^{2}(0,\ell _{\bar{\alpha}})}=o(1).
\label{ex4}
\end{equation}

Define $T_{\bar{\alpha},n}=(u^{\prime}_{\bar{\alpha},n}+a_{%
\bar{\alpha}}v^{\prime}_{\bar{\alpha},n})$ and multiplying (%
\ref{ex2}) by $\beta _{n}^{-\gamma}qT_{\bar{\alpha},n}$ where $q$ is any real function in $H%
^{2}(0,\ell _{\bar{\alpha}})$, we get 
\begin{equation}
Re\int_{0}^{\ell _{\bar{\alpha}}}\mathbf{i}\beta _{n}v_{\bar{%
\alpha },n}q\overline{T_{\bar{\alpha},n}}dx-Re\int_{0}^{\ell _{%
\bar{\alpha}}}T^{\prime}_{\bar{\alpha},n}q\overline{T_{%
\bar{\alpha},n}}dx=o(1).  \label{ex5}
\end{equation}

Using (\ref{ex1}) we have 
\begin{eqnarray}
&&Re\int_{0}^{\ell _{\bar{\alpha}}}\mathbf{i}\beta _{n}v_{\bar{%
\alpha },n}q\overline{T_{\bar{\alpha},n}}dx  \notag \\
&=&-Re\int_{0}^{\ell _{\bar{\alpha}}}v_{\bar{\alpha}%
,n}q(\overline{v^{\prime}_{\bar{\alpha},n}}+
\beta _{n}^{-\gamma}\overline{%
f^{\prime}_{\bar{\alpha},n}})dx+Re\int_{0}^{\ell _{\bar{\alpha}}}\mathbf{i%
}\beta _{n}v_{\bar{\alpha},n}qa_{\bar{\alpha}}%
\overline{v^{\prime}_{\bar{\alpha},n}}dx  \notag \\
&=&-\frac{1}{2}\left[ q(x)\left| v_{\bar{\alpha},n}(x)\right| ^{2}%
\right] _{0}^{\ell _{\bar{\alpha}}}+\frac{1}{2}\int_{0}^{\ell _{%
\bar{\alpha}}}q^{\prime }\left| v_{\bar{\alpha},n}\right|
^{2}dx-Im\int_{0}^{\ell _{\bar{\alpha}}}qa_{\bar{\alpha}}\beta
_{n}v_{\bar{\alpha},n}\overline{v^{\prime}_{\bar{\alpha},n}}%
dx + o(1).  \label{ex6}
\end{eqnarray}
On the other hand, integrating the second term in (\ref{ex5}) by parts,
yields 
\begin{equation}
Re\int_{0}^{\ell _{\bar{\alpha}}}T^{\prime}_{\bar{\alpha}%
,n}q\overline{T_{\bar{\alpha},n}}dx=\frac{1}{2}\left[ q(x)\left| T_{%
\bar{\alpha},n}(x)\right| ^{2}\right] _{0}^{\ell _{\bar{\alpha}%
}}-\frac{1}{2}\int_{0}^{\ell _{\bar{\alpha}}} q^{\prime}\left| T_{%
\bar{\alpha},n}\right| ^{2}dx.  \label{ex7}
\end{equation}
Hence, by substituing (\ref{ex6}) and (\ref{ex7}) into (\ref{ex5}), we
obtain 
\begin{gather}
\frac{1}{2}\int_{0}^{\ell _{\bar{\alpha}}}q^{\prime}\left| v_{%
\bar{\alpha},n}\right| ^{2}dx+\frac{1}{2}\int_{0}^{\ell _{\bar{%
\alpha }}}q^{\prime}\left| T_{\bar{\alpha},n}\right|
^{2}dx-Im\int_{0}^{\ell _{\bar{\alpha}}}qa_{\bar{\alpha}}\beta
_{n}v_{\bar{\alpha},n}\overline{v^{\prime}_{\bar{\alpha},n}}%
dx  \notag \\
-\frac{1}{2}\left( \left[ q(x)\left| v_{\bar{\alpha},n}(x)\right| ^{2}%
\right] _{0}^{\ell _{\bar{\alpha}}}+\left[ q(x)\left| T_{\overline{%
\alpha },n}(x)\right| ^{2}\right] _{0}^{\ell _{\bar{\alpha}}}\right)
=o(1). \label{xsx}
\end{gather}

\begin{lemma} \label{lemd}
The following property holds
\begin{equation}
Im\int_{0}^{\ell _{\bar{\alpha}}}qa_{\bar{\alpha}}\beta _{n}v_{%
\bar{\alpha},n}\overline{v^{\prime}_{\bar{\alpha},n}}dx=o(1).
\label{ex8_3}
\end{equation}
\end{lemma}

\begin{proof}

Since $
\beta _{n}^{\frac{\gamma}{2}}a_{\bar{\alpha}}^{\frac{1}{2}}v^{\prime}_{\bar{\alpha},n%
}\rightarrow 0$ in $L^{2}(0,\ell _{\bar{\alpha}})$ and $qa_{\bar{%
\alpha }}^{\frac{1}{2}}\in L^{\infty }(0,\ell _{\bar{\alpha}}),$ it
suffices to prove that 
\begin{equation}
\beta _{n}^{1-\frac{\gamma}{2}}\left\| a_{\bar{\alpha}}^{\frac{1}{2}}v_{\bar{\alpha}%
,n}\right\| _{L^{2}(0,\ell _{\bar{\alpha}})}=O(1). \label{nvv}
\end{equation}
For this, taking the inner product of (\ref{ex2}) by $\mathbf{i}
\beta _{n}^{1-2 \gamma}a_{\bar{\alpha}}v_{\bar{\alpha},n}$ leads to 
\begin{equation}
\beta _{n}^{2-\gamma}\left\| a_{\bar{\alpha}}^{\frac{1}{2}}v_{\bar{\alpha}%
,n}\right\| _{L^{2}(0,\ell _{\bar{\alpha}})}^{2}=-\mathbf{i}
\beta _{n}^{1-\gamma}\int_{0}^{\ell _{\bar{\alpha}}}T^{\prime}_{\bar{\alpha}%
,n}a_{\bar{\alpha}}\overline{v_{\bar{\alpha},n}}dx-\mathbf{i}
\beta _{n}^{1-2\gamma}\int_{0}^{\ell _{\bar{\alpha}}}g_{\bar{\alpha},n}a_{%
\bar{\alpha}}\overline{v_{\bar{\alpha},n}}dx.  \label{ex8_1}
\end{equation}
Since $a_{\bar{\alpha}} \in L^{\infty }(0,\ell _{\bar{\alpha}})$ and $g_{\bar{%
\alpha },n}\rightarrow 0$ in $L^{2}(0,\ell _{\bar{\alpha}})$ we can
deduce the inequality 
\begin{equation}
-Re(\mathbf{i}\beta _{n}^{1-2\gamma}\int_{0}^{\ell _{\bar{\alpha}}}g_{\bar{%
\alpha },n}a_{\bar{\alpha}}\overline{v_{\bar{\alpha},n}}dx)\leq 
\frac{1}{4}\beta _{n}^{2-\gamma}\left\| a_{\bar{\alpha}}^{\frac{1}{2}}v_{%
\bar{\alpha},n}\right\| _{L^{2}(\omega _{\bar{\alpha}%
})}^{2}+o(1).  \label{ex8_2}
\end{equation}

On the other hand, we have
\begin{eqnarray}
&&-Re \left(\mathbf{i}\beta _{n}^{1-\gamma}\int_{0}^{\ell _{\bar{\alpha}}}T^{\prime}_{\bar{\alpha},n}a_{\bar{\alpha}}\overline{v_{\bar{%
\alpha },n}}dx\right)  \notag \\
&=&-Re\left[ \mathbf{i}\beta _{n}^{1-\gamma}T_{\bar{\alpha},n}(x)a_{\bar{%
\alpha }}(x)\overline{v_{\bar{\alpha},n}}(x)\right] _{0}^{\ell _{%
\bar{\alpha}}}  \notag \\
&&+Re\left[ \mathbf{i}\beta _{n}^{1-\gamma}\int_{0}^{\ell _{\bar{\alpha}}}\left(
a_{\bar{\alpha}}^{\prime }u^{\prime}_{\bar{\alpha},n}%
\overline{v_{\bar{\alpha},n}}+a_{\bar{\alpha}}a_{\bar{%
\alpha }}^{\prime }v^{\prime}_{\bar{\alpha},n}\overline{v_{%
\bar{\alpha},n}}+a_{\bar{\alpha}}u^{\prime}_{\bar{%
\alpha },n}\overline{v^{\prime}_{\bar{\alpha},n}}\right) dx\right]. 
\label{ex81}
\end{eqnarray}
Using (\ref{ex3}) and (\ref{ex4}) we have 
\begin{equation}
Re\left( \mathbf{i}\beta _{n}^{1-\gamma}\int_{0}^{\ell _{\bar{\alpha}}}a_{%
\bar{\alpha}}u^{\prime}_{\bar{\alpha},n}%
\overline{v^{\prime}_{\bar{\alpha},n}}dx\right) =o(1).  \label{ex82}
\end{equation}
Using again (\ref{ex3}) and the fact that $ a^{\prime}_{\bar{\alpha}} \in L^{\infty }(0,\ell _{\bar{\alpha}}),$ we conclude that 
\begin{equation}
Re\left( \mathbf{i}\beta _{n}^{1-\gamma}\int_{0}^{\ell _{\bar{\alpha}}}a_{%
\bar{\alpha}}a^{\prime}_{\bar{\alpha}}v^{\prime}_{\bar{%
\alpha },n}\overline{v_{\bar{\alpha},n}}dx\right) \leq \frac{1}{4}%
\beta _{n}^{2-\gamma}\left\| a_{\bar{\alpha}}^{\frac{1}{2}}v_{\bar{\alpha},n%
}\right\| _{L^{2}(0,\ell _{\bar{\alpha}})}^{2}+o(1).  \label{ex9}
\end{equation}
Now by (\ref{ex2}), we obtain after integrating by parts that
\begin{eqnarray*}
&&Re\left[ \mathbf{i}\beta _{n}^{1-\gamma}\int_{0}^{\ell _{\bar{\alpha}}}a^{\prime}_{\bar{\alpha}}u^{\prime}_{\bar{\alpha},n}\overline{%
v_{\bar{\alpha},n}}dx\right] =Re\left[\beta _{n}^{-\gamma} \int_{0}^{\ell _{\bar{%
\alpha }}}a^{\prime}_{\bar{\alpha}}(v^{\prime}_{\bar{\alpha}%
,n}+\beta _{n}^{-\gamma}f^{\prime}_{\bar{\alpha},n})\overline{v_{\bar{\alpha},n}}%
dx\right]  \\
&=&\frac{1}{2}\left[\beta _{n}^{-\gamma} a^{\prime}_{\bar{\alpha}}(x)\left| v_{\bar{\alpha}%
,n}(x)\right| ^{2}\right] _{0}^{\ell _{\bar{\alpha}}}-\frac{1}{2}%
\beta _{n}^{-\gamma}\int_{0}^{\ell _{\bar{\alpha}}}a^{\prime \prime}_{\bar{\alpha}}\left| v_{\bar{\alpha},n}\right| ^{2}dx+o(1).
\end{eqnarray*}
Furthermore, using that $a^{\prime \prime}_{\bar{\alpha}}\in
L^{\infty }(0,\ell _{\bar{\alpha}})$ and that $v_{\bar{\alpha},n}
$ is bounded, we deduce 
\begin{equation}
Re\left[ \mathbf{i}\beta _{n}^{1-\gamma}\int_{0}^{\ell _{\bar{\alpha}}}a^{\prime}_{\bar{\alpha}}u^{\prime}_{\bar{\alpha},n}\overline{%
v_{\bar{\alpha},n}}dx\right] \leq \frac{1}{2}\left[\beta _{n}^{-\gamma} a^{\prime}_{\bar{\alpha}}(x)\left| v_{\bar{\alpha}%
,n}(x)\right| ^{2}\right] _{0}^{\ell _{\bar{\alpha}}}+O(1).  \label{ex91}
\end{equation}

Combining (\ref{ex82}), (\ref{ex9}), (\ref{ex91}) with (\ref{ex81}),
we get 
\begin{eqnarray}
-Re(\mathbf{i}\beta _{n}^{1-\gamma}\int_{0}^{\ell _{\bar{\alpha}}}T^{\prime}_{%
\bar{\alpha},n}a\overline{v_{\bar{\alpha},n}}dx) &\leq &-Re\left[
\mathbf{i}\beta _{n}^{1-\gamma}T_{\bar{\alpha},n}(x)a_{\bar{\alpha}}(x)%
\overline{v_{\bar{\alpha},n}}(x)\right] _{0}^{\ell _{\bar{\alpha}%
}}  \notag \\
&&+\frac{1}{2}\left[\beta _{n}^{-\gamma} a^{\prime}_{\bar{\alpha}}(x)\left| v_{\bar{\alpha}%
,n}(x)\right| ^{2}\right] _{0}^{\ell _{\bar{\alpha}}}+\frac{1}{4}\beta _{n}^{2-\gamma}\left\|
a_{\bar{\alpha}}^{\frac{1}{2}}v_{\bar{\alpha},n}\right\|_{L^{2}(0,\ell _{\bar{\alpha}})}^{2} + O(1).  
\label{ex92prime}
\end{eqnarray}
Thus, substituting (\ref{ex8_2}) and (\ref{ex92prime}) into (\ref{ex8_1}) leads to
\begin{eqnarray}
\frac{1}{2}\beta _{n}^{2-\gamma}\left\|
a_{\bar{\alpha}}^{\frac{1}{2}}v_{\bar{\alpha},n}\right\|_{L^{2}(0,\ell _{\bar{\alpha}})}^{2} &\leq &-Re\left[
\mathbf{i}\beta _{n}^{1-\gamma}T_{\bar{\alpha},n}(x)a_{\bar{\alpha}}(x)%
\overline{v_{\bar{\alpha},n}}(x)\right] _{0}^{\ell _{\bar{\alpha}%
}}  \notag \\
 &&+\sum_{\bar{\alpha}\in I}\left( a^{\prime}_{\bar{\alpha}}(\ell _{\bar{\alpha}})\left| v_{\bar{\alpha},n}(\ell
_{\bar{\alpha}})\right| ^{2}-a^{\prime}_{\bar{\alpha}}(0)\left|
v_{\bar{\alpha},n}(0)\right| ^{2}\right) +O(1) \label{ex92} 
\end{eqnarray}
Case (i): Here $\underline{a}$ is continuous in all nodes. For $\gamma=0$, it follows from (\ref{ex92}) that
\begin{equation}
\sum_{\bar{\alpha}\in I}\beta _{n}^{2}\left\| a_{\bar{\alpha}}^{\frac{1}{2}%
}v_{\bar{\alpha},n}\right\| _{L^{2}(0,\ell _{\bar{%
\alpha }})}^{2}\leq 2\sum_{\bar{\alpha}\in I}\left( a^{\prime}_{\bar{\alpha}}(\ell _{\bar{\alpha}})\left| v_{\bar{\alpha},n}(\ell
_{\bar{\alpha}})\right| ^{2}-a^{\prime}_{\bar{\alpha}}(0)\left|
v_{\bar{\alpha},n}(0)\right| ^{2}\right) +O(1).  \label{ex96}
\end{equation}
We have used the continuity condition of $\underline{v}$ and $\underline{a}$
and the compatibilty condition (\ref{ch8eq38}) at inner nodes and the Dirichlet
condition of $\underline{u}$ and $\underline{v}$ at externel nodes.

To conclude, notice that from the property (P) we
deduce that 
\begin{equation*}
\sum_{\bar{\alpha}\in I}\left( a^{\prime}_{\bar{\alpha}}(\ell _{%
\bar{\alpha}})\left| v_{\bar{\alpha},n}(\ell _{\bar{\alpha}%
})\right| ^{2}-a^{\prime}_{\bar{\alpha}}(0)\left| v_{\bar{\alpha}%
,n}(0)\right| ^{2}\right) \leq 0.
\end{equation*}
Then, (\ref{ex96}), yields 
\begin{equation*}
\beta _{n}^{2}\left\| a_{\bar{\alpha}}^{\frac{1}{2}}v_{\bar{\alpha}%
,n}\right\| _{L^{2}(0,\ell _{\bar{\alpha}})}^{2}=O(1)
\end{equation*}
for every $\bar{\alpha} \in I,$ and the proof of Lemma \ref{lemd} is complete for case (i).

Case (ii): Recall that here the function $\underline{a}$ is not continuous at some internal nodes. 
For $\gamma=\frac{1}{2}$, we want estimate the first term in the right hand side of (\ref{ex92}). To do this it suffices to estimate $Re (\mathbf{i}\beta _{n}^{1-\gamma}T_{\bar{\alpha},n}(x)a_{\bar{%
\alpha }}(x_{\bar{\alpha}})\overline{v_{\bar{\alpha},n}}(x))$ at an inner node $x=x_{\bar{\alpha}}$ when $a_{\bar{\alpha}}(x_{\bar{\alpha}})\neq 0$.
For simplicity and without loss of generality we suppose that $x_{\bar{\alpha}}$ is the end of $e_{\bar{\alpha},n}$ identified to $0$ via $\pi_{\bar{\alpha}}$.

Since $a_{\bar{\alpha}}$ is continuous on $[0,\ell_{\bar{\alpha}}]$, there exists a positive number $k_{\bar{\alpha}}<\ell_{\bar{\alpha}}$ such that $a_{\bar{\alpha}}(x)\neq 0$ on $[0,\ell_{\bar{\alpha}}]$. We first prove 
\begin{equation}
\beta _{n}\left\| v_{\bar{\alpha },n}\right\|_{L^{2}(0,k _{\bar{\alpha }})} ^{2}=o(1). \label{c2}
\end{equation}

We need the following Gagliardo-Nirenberg inequality \cite{nir} in estimation:

There exists two positives constants $C_{1}$ and $C_{2}$ such that, for any $w$ in $H^{1}(0,k_{\bar{\alpha}})$,
\begin{equation} 
\left\| w\right\| _{L^{\infty }(0,k_{\bar{\alpha}})}\leq C_{1}\left\| w\right\|_{L^{2}(0,k _{\bar{\alpha}})}^{\frac{1}{2}}\left\| w^{\prime}\right\|_{L^{2}(0,k _{\bar{\alpha}})}^{\frac{1}{2}}+C_{2}\left\|w\right\|_{L^{2}(0,k _{\bar{\alpha}})}. \label{gn}
\end{equation}

Multiplying (\ref{ex2}) by $\mathbf{i}\beta _{n}^{-1}v_{\bar{\alpha },n}$ in $ L^{2}(0,\ell _{\bar{\alpha }})$  and
integrating by parts, we obtain
\begin{equation*}
\beta _{n}^{\frac{1}{2}}\left\| v_{\bar{\alpha },n}\right\| ^{2}_{L^{2}(0,\ell _{\bar{\alpha }})}=-%
\mathbf{i}\beta _{n}^{-\frac{1}{2}}\left[ T_{\bar{\alpha },n}(x)\overline{%
v_{\bar{\alpha },n}}(x)\right] _{0}^{k _{\bar{\alpha }}}+\mathbf{i}%
\beta _{n}^{-\frac{1}{2}}\int_{0}^{k _{\bar{\alpha }}}T_{\bar{%
\alpha },n}\overline{v^{\prime}_{\bar{\alpha },n}}dx+o(1).
\end{equation*}
By (\ref{ex3}) and (\ref{ex4}) we have $\mathbf{i}\beta _{n}^{-\frac{1}{2}}\int_{0}^{k _{\bar{%
\alpha }}}T_{\overline{\alpha },n}\overline{v^{\prime}_{\bar{\alpha 
},n}}dx=o(1).$ 

Using Gagliardo-Nerenberg inequality (\ref{gn}), (\ref{ex3}), (\ref{ex4}) and the boundedness of $v_{\bar{\alpha 
},n}$,
\begin{eqnarray*}
\left\| v_{\bar{\alpha },n}\right\| _{L^{\infty
}(0,k _{\bar{\alpha}})} &\leq &C_{1}\left\| v_{\bar{\alpha },n}\right\|_{L^{2}(0,k _{\bar{\alpha }})} ^{\frac{1%
}{2}}\left\| v^{\prime}_{\bar{\alpha },n}\right\|_{L^{2}(0,k _{\bar{\alpha }})} ^{\frac{1}{2}}+C_{2}\left\|
v_{\bar{\alpha },n}\right\|_{L^{2}(0,k _{\bar{\alpha }})}=O(1) , \\
\beta _{n}^{-\frac{3}{8}}\left\| T_{\bar{\alpha },n}\right\| _{L^{\infty
}(0,k _{\bar{\alpha}})} &\leq
&C_{1}\left\| \beta _{n}^{\frac{1}{4}}T_{\bar{\alpha },n}\right\|_{L^{2}(0,k _{\bar{\alpha }})} ^{\frac{1}{2}}\left\|
\beta _{n}^{-1}T^{\prime}_{\bar{\alpha },n}\right\|_{L^{2}(0,k _{\bar{\alpha }})} ^{\frac{1}{2}}+C_{2}\beta _{n}^{-%
\frac{3}{8}}\left\| T_{\bar{\alpha },n}\right\|_{L^{2}(0,k _{\bar{\alpha }})}=o(1) .
\end{eqnarray*}
It follows that
$\mathbf{i}\beta _{n}^{-\frac{1}{2}}\left[ T_{\bar{\alpha },n}(x)\overline{%
v_{\bar{\alpha },n}}(x)\right] _{0}^{k _{\bar{\alpha }}}=o(1)$ and
then $\beta _{n}^{\frac{1}{2}}\left\| v_{\bar{\alpha },n}\right\|_{L^{2}(0,k _{\bar{\alpha }})}
^{2}=o(1).$

Then, we multiply (\ref{ex2}) by $\mathbf{i}\beta _{n}^{-\frac{1}{2}}v_{\bar{%
\alpha },n}$ and we repeat exactly the same strategy as before, using (\ref{ex2}) and $\beta _{n}^{\frac{1}{2}}\left\| v_{\bar{\alpha },n}\right\|
^{2}=o(1)$, we obtain (\ref{c2}).

We are now ready to estimate $-Re (\mathbf{i}\beta _{n}^{\frac{1}{%
2}} T_{\bar{\alpha },n}(0)\overline{v_{\bar{\alpha },n}}(0))
.$

Applying Gagliardo-Nerenberg inequality (\ref{gn}) to $w=v_{\bar{\alpha },n}$ we obtain, using (\ref{ex3}), 
\begin{eqnarray*}
\beta _{n}^{\frac{1}{2}}\left\| v_{\bar{\alpha },n}\right\| _{L^{\infty
}(0,k _{\bar{\alpha}})} &\leq &C_{1}\left\|
\beta _{n}^{\frac{3}{4}}v_{\bar{\alpha },n}\right\|_{L^{2}(0,k _{\bar{\alpha }})} ^{\frac{1}{2}}\left\| \beta _{n}^{%
\frac{1}{4}}v^{\prime}_{\bar{\alpha },n}\right\|_{L^{2}(0,k _{\bar{\alpha }})} ^{\frac{1}{2}}+C_{2}\beta _{n}^{\frac{%
1}{2}}\left\| v_{\bar{\alpha },n}\right\| _{L^{2}(0,k_{\bar{\alpha }})}  \\
&\leq &o(1)+ \left\| \beta _{n}^{\frac{3}{4}}v_{\bar{\alpha },n}\right\| o(1).
\end{eqnarray*}
Using again the Gagliardo-Nerenberg inequality (\ref{gn}) with $w=T_{\bar{\alpha },n}$,
\begin{eqnarray*}
\left\| T_{\bar{\alpha },n}\right\| _{L^{\infty
}(0,k _{\bar{\alpha}})} &\leq &C_{1}\left\| \beta _{n}^{\frac{1}{4}%
}T_{\bar{\alpha },n}\right\|_{L^{2}(0,k _{\bar{\alpha }})} ^{\frac{1}{2}}\left\| \beta _{n}^{-\frac{1}{4}}T^{\prime}_{\bar{\alpha },n}\right\| ^{\frac{1}{2}}+C_{2}\left\| T_{\bar{\alpha },n}\right\|_{L^{2}(0,k _{\bar{\alpha }})}  \\
&\leq &o(1)\left\| \beta _{n}^{\frac{1}{4}%
}T^{\prime}_{\bar{\alpha },n}\right\|^{\frac{1}{2}}_{L^{2}(0,k _{\bar{\alpha }})} + o(1)  \\
&\leq &o(1)+o(1) \left\| \beta _{n}^{\frac{3}{4}%
}v_{\bar{\alpha },n}\right\|_{L^{2}(0,k _{\bar{\alpha }})} \, . 
\end{eqnarray*}
Here, we have used (\ref{ex2}),(\ref{ex3}) and (\ref{ex4}). Then
\begin{equation}
\vert Re (\mathbf{i}\beta _{n}^{\frac{1}{%
2}} T_{\bar{\alpha },n}(0)\overline{v_{\bar{\alpha },n}}(0))\vert
 \leq \beta _{n}^{\frac{1}{2}}\left\| v_{\bar{\alpha },n}\right\| _{L^{\infty
}(0,\ell _{\bar{\alpha}})}\left\|
T_{\bar{\alpha },n}\right\| _{L^{\infty
}(0,\ell _{\bar{\alpha}})}\leq \frac{1}{4}\beta _{n}^{\frac{3}{2}}\left\| %
v_{\bar{\alpha },n}\right\|_{L^{2}(0,k _{\bar{\alpha }})} ^{2}+o(1) \label{bb3}
\end{equation}
and
\begin{equation}
-Re \left[ \mathbf{i}\beta _{n}^{\frac{1}{%
2}} T_{\bar{\alpha },n}(x)\overline{v_{\bar{\alpha },n}}(x)\right]_{0}^{k _{\bar{\alpha }}} 
 \leq 2\beta _{n}^{\frac{1}{2}}\left\| v_{\bar{\alpha },n}\right\| _{L^{\infty
}(0,\ell _{\bar{\alpha}})}\left\|
T_{\bar{\alpha },n}\right\| _{L^{\infty
}(0,\ell _{\bar{\alpha}})}\leq \frac{1}{2}\beta _{n}^{\frac{3}{2}%
}\left\| v_{\bar{\alpha },n}\right\|_{L^{2}(0,k _{\bar{\alpha }})} ^{2}+o(1) \label{bb2}
\end{equation}

Multiplying (\ref{ex2}) by $\mathbf{i}v_{\bar{\alpha },n}$ in $ L^{2}(0,k _{\bar{\alpha }})$  and
integrating by parts, we obtain
\begin{equation}
\beta _{n}^{\frac{3}{2}}\left\| v_{\bar{\alpha },n}\right\| ^{2}_{L^{2}(0,k _{\bar{\alpha }})}=-%
\mathbf{i}\beta _{n}^{\frac{1}{2}}\left[ T_{\bar{\alpha },n}(x)\overline{%
v_{\bar{\alpha },n}}(x)\right] _{0}^{k _{\bar{\alpha }}}+\mathbf{i}%
\beta _{n}^{\frac{1}{2}}\int_{0}^{k _{\bar{\alpha }}}T_{\bar{%
\alpha },n}\overline{v^{\prime}_{\bar{\alpha },n}}dx+o(1). \label{bb1}
\end{equation}
Using (\ref{ex3}) and (\ref{ex4}), the second term on the left hand side of (\ref{bb1}) converge to zero. We conclude, using (\ref{bb2}) that
\begin{equation*}
\beta _{n}^{\frac{3}{2}}\left\| v_{\bar{\alpha },n}\right\| ^{2}_{L^{2}(0,k _{\bar{\alpha }})}=o(1).
\end{equation*}
Return back to (\ref{bb3}) which yields
\begin{equation*}
-Re (\mathbf{i}\beta _{n}^{\frac{1}{%
2}} T_{\bar{\alpha },n}(0)\overline{v_{\bar{\alpha },n}}(0))
 =o(1).
\end{equation*} 
We obtain the same result if we suppose that $x_{\bar{\alpha}}$ is the end of $e_{\bar{\alpha},n}$ identified to $\ell_{\bar{\alpha}} $ via $\pi_{\bar{\alpha}}$, that is
\begin{equation*}
-Re (\mathbf{i}\beta _{n}^{\frac{1}{%
2}} T_{\bar{\alpha },n}(\ell_{\bar{\alpha}})\overline{v_{\bar{\alpha },n}}(\ell_{\bar{\alpha}}))
 =o(1),
\end{equation*}
and we then conclude that the first term on the right hand side of (\ref{ex92}) converge to zero.

Now, summing over $\bar{\alpha} \in I$ in (\ref{ex92}) by taking into account the estimate of  $-Re \left[ \mathbf{i}\beta _{n}^{\frac{1}{%
2}} T_{\bar{\alpha },n}(x)\overline{v_{\bar{\alpha },n}}(x)\right]_{0}^{\ell _{\bar{\alpha }}}$ and the inequality in (P), we obtain that
\begin{equation*}
\sum_{\bar{\alpha}\in I}\beta _{n}^\frac{1}{2}\left\| a_{\bar{\alpha}}^{\frac{1}{2}%
}\beta _{n}v_{\bar{\alpha},n}\right\| _{L^{2}(0,\ell _{\bar{%
\alpha }})}^{2}=O(1),
\end{equation*}

then 
\begin{equation*}
\beta _{n}^\frac{3}{2}\left\| a_{\bar{\alpha}}^{\frac{1}{2}}v_{\bar{\alpha}%
,n}\right\| _{L^{2}(0,\ell _{\bar{\alpha}})}^{2}=O(1)
\end{equation*}
for every $\bar{\alpha} \in I,$ and the proof of Lemma \ref{lemd} is complete for case (ii).
\end{proof}
Return back to the general case. Substituting (\ref{ex8_3}) in (\ref{xsx}) leads to 
\begin{equation}
\frac{1}{2}\int_{0}^{\ell _{\bar{\alpha}}}  q^{\prime} \, \left| v_{%
\bar{\alpha},n}\right| ^{2}dx+\frac{1}{2}\int_{0}^{\ell _{\bar{%
\alpha }}} q^{\prime}\left| T_{\bar{\alpha},n}\right| ^{2}dx-\frac{1}{2%
}\left[ q(x)\left( \left| v_{\bar{\alpha},n}(x)\right| ^{2}+\left| T_{%
\bar{\alpha},n}(x)\right| ^{2}\right) \right] _{0}^{\ell _{\bar{%
\alpha }}}=o(1)  \label{ex10}
\end{equation}

for every $\bar{\alpha}\in I.$

Let $\bar{\alpha}\in I$ such that $e_{\bar{\alpha}}$ is a K-V
string. First, note that from (\ref{nvv}), we deduce that 
\begin{equation*}
\left\| a_{\bar{\alpha}}^{\frac{1}{2}}v_{\bar{\alpha},n}\right\|
_{L^{2}(0,\ell _{\bar{\alpha}})}^{2}=o(1).
\end{equation*}
Then, we take $q(x)=\int_{0}^{x}a_{\bar{\alpha}}(s)ds$ in (\ref{ex10}%
) to obtain
\begin{equation}
\frac{1}{2}\int_{0}^{\ell _{\bar{\alpha}}}a_{\bar{\alpha}}\left|
T_{\bar{\alpha},n}\right| ^{2}dx-\frac{1}{2}\left( \int_{0}^{\ell _{%
\bar{\alpha}}}a_{\bar{\alpha}}(s)ds\right) \left( \left| v_{%
\bar{\alpha},n}(\ell _{\bar{\alpha}})\right| ^{2}+\left| T_{%
\bar{\alpha},n}(\ell _{\bar{\alpha}})\right| ^{2}\right) =o(1).
\label{ex8_5}
\end{equation}

Since
$
\frac{1}{2}\int_{0}^{\ell _{\bar{\alpha}}}a_{\bar{\alpha}}\left|
T_{\bar{\alpha},n}\right| ^{2}dx=o(1)
$
and $\int_{0}^{\ell _{\bar{\alpha}}}a_{%
\bar{\alpha}}(s)ds>0$,
then (\ref{ex8_5}) implies 
\begin{equation}
\left| T_{\bar{\alpha},n}(\ell _{\bar{\alpha}})\right|
^{2}+\left| v_{\bar{\alpha},n}(\ell _{\bar{\alpha}})\right|
^{2}=o(1).  \label{exx}
\end{equation}

Therefore (\ref{ex10}) can be rewritten as 
\begin{eqnarray}
\frac{1}{2}\int_{0}^{\ell _{\bar{\alpha}}} q^{\prime} \, \left| v_{%
\bar{\alpha},n}\right| ^{2}dx+\frac{1}{2}\int_{0}^{\ell _{\bar{%
\alpha }}} q^{\prime} \, \left| T_{\bar{\alpha},n}\right| ^{2}dx  \notag \\
+\frac{1}{2}\left( q(0)\left| v_{\bar{\alpha},n}(0)\right|
^{2}+q(0)\left| T_{\bar{\alpha},n}(0)\right| ^{2}\right) &=&o(1).
\label{ex16}
\end{eqnarray}
Taking $q=x$ in (\ref{ex16}) implies that $\left\| v_{\bar{\alpha}%
,n}\right\| _{L^{2}(0,\ell _{\bar{\alpha}})}=o(1)$ and $\left\| T_{%
\bar{\alpha},n}\right\| _{L^{2}(0,\ell _{\bar{\alpha}})}=o(1).$
Moreover, $\left\| u^{\prime}_{\bar{\alpha},n}\right\|
_{L^{2}(0,\ell _{\bar{\alpha}})}=\left\| T_{\bar{\alpha},n}-a_{%
\bar{\alpha}}v_{\bar{\alpha},n}\right\| _{L^{2}(0,\ell _{%
\bar{\alpha}})}=o(1).$

By letting $q=\ell _{\bar{\alpha}}-x$ in (\ref{ex16}) and by taking
into account the convergence of $v_{\bar{\alpha},n}$ and $T_{\bar{%
\alpha },n}$ in $L^{2}(0,\ell _{\bar{\alpha}}),$ we get 
\begin{equation}
v_{\bar{\alpha},n}(0)=o(1)\text{ and }T_{\bar{\alpha},n}(0)= o(1).
\label{ex101}
\end{equation}
Finally, notice that (\ref{exx}) signifies that 
\begin{equation}
v_{\bar{\alpha},n}(\ell _{\bar{\alpha}})=o(1)\text{ and }T_{%
\bar{\alpha},n}(\ell _{\bar{\alpha}})= o(1).  \label{ex102}
\end{equation}
To conclude, it suffices to prove that 
\begin{equation}
\left\| v_{\bar{\alpha},n}\right\| _{L^{2}(0,\ell _{\bar{\alpha}%
})}=o(1)\text{ and }\left\| u^{\prime}_{\bar{\alpha},n}\right\|
_{L^{2}(0,\ell _{\bar{\alpha}})}=\left\| T_{\bar{\alpha}%
,n}\right\| _{L^{2}(0,\ell _{\bar{\alpha}})}=o(1)  \label{ex103}
\end{equation}
for every $\bar{\alpha}\in I$ such that $e_{\bar{\alpha}}$ is
purely elastic.

To do this, starting by a string $e_{\bar{\alpha}}$ attached at one
end to only K-V strings. Using continuity condition of $\underline{v}$ and
the compatibility condition at inner nodes, implies that $e_{\bar{%
\alpha }}$ satisfies (\ref{ex101}) or (\ref{ex102}). Moreover, by taking $%
q=1 $ in (\ref{ex10}), we conclude that $e_{\bar{\alpha}}$ satisfies (%
\ref{ex101}) and (\ref{ex102}). Then using again (\ref{ex10}) with $q=x,$ we
deduce that (\ref{ex103}) is satisfied by $e_{\bar{\alpha}}.$ We
iterate such procedure on each maximally connected subgraph of purely elastic
strings (from leaves to the root).

Thus $\left\| (\underline{u}_{n},\underline{v}_{n})\right\| _{\mathcal{H}%
}=o(1),$ which contradicts the hypoyhesis $\left\| (\underline{u}_{n},%
\underline{v}_{n})\right\| _{\mathcal{H}}=1.$
\end{proof}

\begin{remark}
\begin{enumerate}
\item If for every $\bar{\alpha } \in I $, $a_{\bar{\alpha }}$ is continuous on $[0,\ell_{\bar{\alpha }}]$ and not vanish in such interval then we don't need the property (P) in the Theorem \ref{th}. 

Indeed (P) is used only to estimate $$-Re\left(\mathbf{i}\beta _{n}^{1-\gamma}\int_{0}^{\ell _{\bar{\alpha}}}T^{\prime}_{\bar{\alpha},n}a_{\bar{\alpha}}\overline{v_{\bar{%
\alpha },n}}dx\right)$$ in \rfb{ex8_1}, according to $\beta _{n}^{1-\frac{\gamma}{2}}\left\| a_{\bar{\alpha}}^{\frac{1}{2}}v_{\bar{\alpha}%
,n}\right\| _{L^{2}(0,\ell _{\bar{\alpha}})}$.

This is equivalent to estimate$$-Re\left(\mathbf{i}\beta _{n}^{1-\gamma}\int_{0}^{\ell _{\bar{\alpha}}}T^{\prime}_{\bar{\alpha},n}\overline{v_{\bar{%
\alpha },n}}dx\right)$$ according to $\beta _{n}^{1-\frac{\gamma}{2}}\left\|v_{\bar{\alpha}%
,n}\right\|_{L^{2}(0,\ell _{\bar{\alpha}})}$~:
$$
- \, Re \left( {\bf i} \beta_n^{1-\gamma} \, \int_0^{\ell_{\bar{\alpha}}} 
 T^{\prime}_{\bar{\alpha},n}\overline{v_{\bar{\alpha },n}}dx\right) =
- Re \left[{\bf i} \beta_n^{1-\gamma} \, T_{\bar{\alpha},n} \, \overline{v_{\bar{\alpha },n}}\right]_0^{\ell_{\bar{\alpha}}} + Re\left(\mathbf{i}\beta _{n}^{1-\gamma}\int_{0}^{\ell _{\bar{\alpha}}}T_{\bar{\alpha},n} \overline{v^{\prime}_{\bar{%
\alpha },n}}dx\right) =
$$
$$
- Re \left[{\bf i} \beta_n^{1-\gamma} \, T_{\bar{\alpha},n}(x) \, \overline{v_{\bar{\alpha },n}}(x)\right]_0^{\ell_{\bar{\alpha}}} + o(1)
$$
as in case (ii) (proof of Theorem \ref{th}) we prove without using (P) 
that 
$$
- Re \left[{\bf i} \beta_n^{1-\gamma} \, T_{\bar{\alpha},n}(x) \, \overline{v_{\bar{\alpha },n}}(x)\right]_0^{\ell_{\bar{\alpha}}}  \leq \frac{\beta _{n}^{2-\gamma}}{4} \, \left\|v_{\bar{\alpha}%
,n}\right\|^2_{L^{2}(0,\ell _{\bar{\alpha}})} + o(1).
$$

\item We find here the particular cases studied in \cite{liu11,liu111,riv,hassine,liu1}. Note that concerning the result of polynomial stability in \cite{riv,hassine} the authors proved that the $\frac{1}{t^2}$ decay rate of solution is optimal when the damping coefficient is a characteristic function.
\end{enumerate}

\end{remark}

\section{Further comments: graph case}
In this section we want generalize the previous results to a general graph. Then we suppose that $\mathcal {T}$ is a connected graph $\mathcal {G}$  (then $\mathcal {G}$ can contains some circuits).

\begin{center}
\includegraphics[scale=0.80]{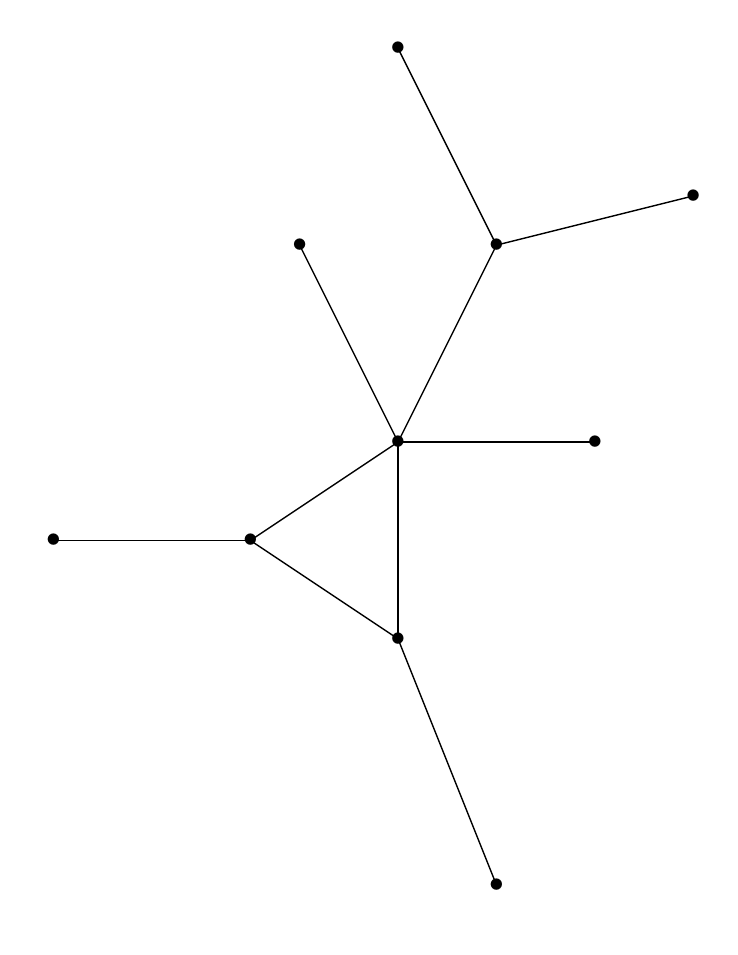} 
\captionof{figure}{A Graph} 
\end{center}

We conserve the same notations as in $\mathcal{T}$, we just replace $\mathcal{S}$ by $\mathcal{S}^{\prime}$ which is the set of all the external nodes, and $I_{\mathcal{S^{\prime}}}$  denote the set of such nodes. Then $I=I_{\mathcal{M}}\cup I_{\mathcal{S^{\prime}}}$. We denote by $J$ the set $%
\{1,...,N\}$ and for $k\in I$ we will denote by $J_{k}$ the set of indices
of edges adjacent to $s_{k}$. If $k\in I_{\mathcal{S}^{\prime}},$ then the index of
the unique element of $J_{k}$ will be denoted by $j_{k}.$

We suppose that the graph is directed, then we need to define the incidence matrix $D=(d_{kj})_{p\times N}, p = |I|,$ as follows,
\begin{equation*}
d_{kj}=\left\{ 
\begin{tabular}{l}
$1$ if $\pi _{j}(\ell _{j})=s_{k},$ \\ 
$-1$ if $\pi _{j}(0)=s_{k},$ \\ 
$0$ otherwise,
\end{tabular}
\right. 
\end{equation*}
The system \rfb{ch8eq18}-\rfb{ch8eq38} is rewritten as follows
\begin{equation}
\label{ch8eq18_1}
\frac{\partial^2
u_{j}}{\partial t^2}(x,t) - \, \frac{\partial }{\partial x} \left( \frac{\partial u_{j}}{\partial x} + a_{j}(x) \, 
\frac{\partial^2 u_{j}}{\partial x \partial t} \right) (x,t) = 0, \quad
0<x<\ell_{j},\ t>0, \, j \in J, 
\end{equation} 
\begin{equation}
u_{j_{k}}(s_{k},t)=0,
\, k \in I_{{\mathcal S}^{\prime}}, \, t >0, 
\end{equation} 
\begin{equation}
u_{j}(s_{k},t)=
u_{l}(s_{k},t), \quad t>0, \, j,l\in J_{k},\;\;k\in I_{\mathcal{M}},
\end{equation} 
\begin{equation}
\label{ch8cont_1}
 \ds \sum\limits_{j\in J_{k}}d_{kj} \left(
\frac{\partial u_{j}}{\partial x}(s_{k},t)+a_{j}(s_{k}) \frac{\partial^2 u_{j}}{\partial x \partial t}(s_{k},t) \right) =
0, 
\quad t>0, \, k \in I_{{\mathcal M}}, 
\end{equation} 
\begin{equation} 
\label{ch8eq38_1}
u_{j}(x, 0) = u^0_{j}(x), \, \frac{\partial u_{j}}{\partial t} (x, 0) = u^1_{j}(x), \quad
0<x<\ell_{j}, \, j \in J. 
\end{equation} 
As in the case of a tree, we suppose that $\mathcal{G}$ contains at least a K-V edge and that every maximal subgraph of purely elastic edges is (a tree), the
leaves of which K-V edges are attached. Furthermore, we suppose that $\mathcal{S}^{\prime}\neq \emptyset.$

Finally the property (P) is rewritten as follows,
\begin{equation*}
(P')\;\;\;\forall j\in J,\;\;a_{j}^{\prime },a_{j}^{\prime \prime }\in L^{\infty }(0,\ell _{j})\;\; \text{and}\;\; \;\;\forall k\in I_{\mathcal{M}%
},\;\; \ds \sum\limits_{j\in J_{k}}d_{kj}a_{j}^{\prime }(s_{k})\geq 0. 
\end{equation*}

Under the property (P') the system \rfb{ch8eq18_1}-\rfb{ch8eq38_1} is polynomially stable, and it is exponentially stable if and only if $\underline{a}$ is continuous at each inner nodes, i.e., $u_{j}(s_{k})=
u_{l}(s_{k}), \quad t>0, \, ;j,l\in J_{k},\;\;k\in I_{\mathcal{M}}$.

\end{document}